\documentclass[reqno]{amsart}
\usepackage{amsmath}
\usepackage{amsthm}
\usepackage{amsfonts}
\usepackage{amssymb}
\usepackage{graphicx}
\usepackage{mathrsfs}

\title[$L_p$-DISCREPANCIES OF POINT DISTRIBUTIONS]{BOUNDS FOR $L_p$-DISCREPANCIES OF POINT DISTRIBUTIONS IN COMPACT METRIC MEASURE SPACES}

\author{M.M. SKRIGANOV}

\address{St. Petersburg Department, Steklov Mathematical Institute, Russian Academy of Sciences}

\email{maksim88138813@mail.ru}



\numberwithin{equation}{section}

\newtheorem{theorem}{Theorem}[section]
\newtheorem{lemma}{Lemma}[section]
\newtheorem{proposition}{Proposition}[section]
\newtheorem{corollary}{Corollary}[section]
\theoremstyle{remark}

\theoremstyle{remark}

\def\dd{\mathrm{d}}

\def\Cc{\mathbb{C}}

\def\Ee{\mathbb{E}}
\def\Ff{\mathbb{F}}

\def\Hh{\mathbb{H}}

\def\Oo{\mathbb{O}}

\def\Rr{\mathbb{R}}

\def\DDD{\mathcal{D}}
\def\EEE{\mathcal{E}}

\def\III{\mathcal{I}}

\def\LLL{\mathcal{L}}
\def\MMM{\mathcal{M}}

\def\PPP{\mathcal{P}}
\def\QQQ{\mathcal{Q}}

\DeclareMathOperator{\diam}{diam}

\renewcommand{\le}{\leqslant}
\renewcommand{\ge}{\geqslant}

\begin{document}

\begin{abstract}

Upper bounds for the $L_p$-discrepancies of point distributions 
in compact metric measure spaces for $0<p\le\infty$ have been established
in the paper \cite{ref5*} 
by Brandolini, Chen, Colzani, Gigante and Travaglini.
In the present paper we show
that such bounds can be established in a much more general
situation under very simple 
conditions on the volume of metric balls as a function of radii.
\end{abstract}

\maketitle

\thispagestyle{empty}

%%%%%%%%%%
%
% SECTION 1
%
%%%%%%%%%%

\section{Introduction}\label{sec1}

Let $\MMM$ be a compact metric measure space with a fixed metric $\theta$ and a finite non-negative Borel measure~$\mu$, normalized, for convenience, by 
\begin{equation}
\mu(\MMM)=1, \quad \diam \MMM =1,
\label{eq1.1}
\end{equation}
where $\diam \EEE = \sup \{\theta(y_1,y_2), y_1,y_2 \in\EEE\}$ denotes 
the diameter of a set $\EEE\subseteq\MMM$.

Since $\MMM$ is connected and satisfies \eqref{eq1.1}, the set of values of $\theta$
coincides with the interval $\III=[0,1]$. We write $B(y,r)=\{x:\theta(x,y)<r\}$ 
for the ball in $\MMM$ of radius $r\in \III$ centered at~ $y\in\MMM$
and of volume $v(y,r)=\mu(B(y,r))$. We can conveniently write $B(y,r)=\emptyset$ and
$v(y,r)=0$ if $r\le 0$ and $B(y,r)=\MMM$ and $v(y,r)=1$ if $r\ge 1$.

The \textit{local discrepancy} of an $N$-point subset $\DDD_N \subset \MMM$ 
(distribution) in a metric ball $B(y,r)$ is defined by
\begin{align}
L[B(y,r),\DDD_N]
&
=\#(B(y,r)\cap \DDD_N)-N v(y,r))
\nonumber
\\
&
=\sum_{x\in\DDD_N}L(y,r,x),
\label{eq1.4}
\end{align}
where
\begin{equation}
L(y,r,x)=\chi(B(y,r),x)-v(y,r),
\label{eq1.5}
\end{equation}
and $\chi(\EEE,x)$ denotes the characteristic function of a subset $\EEE\subset\MMM$.

For $0<p<\infty$, the $L_p$-\textit{discrepancy} is defined by
\begin{equation}
\LLL_p[\xi,\DDD_N]=\LLL_p[\MMM,\xi,\DDD_N]=
\left(\int\!\!\!\!\int_{\MMM\times\III}\vert L[y,r,\DDD_N]\vert^p\,\dd\mu(y)\,
\dd\xi(r)\right)^{1/p},
\label{eq1.6}
\end{equation}
where $\xi$ is a finite (non-negative) measure on $\III$ normalized by $\xi(\III)=1$.

For $p=\infty$, we put
\begin{equation}
\LLL_{\infty}[\DDD_N]=\LLL_{\infty}[\MMM,\DDD_N]=\sup_{y,r}L[y,r,\DDD_N],
\label{eq1.7}
\end{equation}
where the supremum is taken over all balls $B(y,r)\subseteq\MMM$.

We introduce also the following extremal discrepancies
\begin{equation}
\left.
\begin{aligned}
\lambda_p[\xi,N]=\lambda_p[\MMM,\xi,N]=\inf_{\DDD_N}\LLL_p[\xi,\DDD_N],
\\  
\lambda_{\infty}[N]=\lambda_{\infty}[\MMM,N]=\inf_{\DDD_N}\LLL_{\infty}[\DDD_N],
\\
\end{aligned}
\quad \right\}
\label{eq1.8}
\end{equation}
where the infimum is taken over all $N$-point subsets $\DDD_N\subset\MMM$.

Point distributions on the spheres $S^d$ have been studied by many authors, see 
the surveys ~\cite{ref2, ref6} and references therein. We mention only
a few results known at present. First of all, we have the following
two-side bounds
\begin{equation}
N^{\frac12 - \frac{1}{2d}}\,\,\lesssim \,\,\lambda_2[S^d,\xi,N]\,\, 
\lesssim\,\, N^{\frac12 - \frac{1}{2d}},
\label{eq1.9*}
\end{equation}
where the measure 
$d\xi(r)=\frac{\pi}{2}\sin(\pi r)dr,\,\, r\in\III$ (for this measure
the geodesic balls $B(y,r)\subset S^d$ coincide with the spherical caps).
The upper bound in \eqref{eq1.9*} was proved by Alexander~\cite{ref1} 
and Stolarsky~\cite{ref13} and the lower bound by Beck ~\cite{ref3}, 
see also~\cite[Theorem~24A and Corollary~24C]{ref4}. Furthermore, 
Beck, see~\cite[Theorem~24D]{ref4}, established the bound
\begin{equation}
\lambda_{\infty}[S^d,N]\lesssim N^{\frac12-\frac{1}{2d}}\,(\log N)^{1/2}.
\label{eq1.12*}
\end{equation}
The constants implicit in the symbol $\lesssim$ are independent of $N$.

The two-side bounds \eqref{eq1.9*} were extended to all compact Riemannian symmetric manifolds of rank one (two-point homogeneous spaces)
and arbitrary absolutely continuous measures $\xi$ on $\III$, 
see~\cite[Theorem~2.2]{ref12}.  
Recall that these manifolds are the spheres $S^d$, the real, complex and quaternionic projective spaces $\Ff P^n, \Ff=\Rr, \Cc, \Hh$, and the octonionic projective 
plane $\Oo P^2$, see, for example, ~\cite{ref10}. Recall that any Riemannian manifold can be thought of as 
a mertric measure space with respect to the Riemannian distance and measure
(normalized, if needed by \eqref{eq1.1}). 

Notice that the upper bound in \eqref{eq1.9*}
holds for arbitrary compact $d$-rectifiable spaces  $\MMM$ and any measure $\xi$, 
see~\cite{ref11}. 
At the same time,  the lower bound in \eqref{eq1.9*} fails, even for 
the spheres $S^d$, if  the measure $\xi$ is singular. The corresponding example 
can be found in~\cite{ref5, ref11, ref12}. In this example the discrepancy
$\lambda_p[\xi,N]$ is bounded from above by a constant independent of $N$ and $p$.

The $L_p$-discrepancies for any $p$ and for general compact metric measure spaces 
were first estimated in the paper \cite{ref5*} 
by Brandolini, Chen, Colzani, Gigante and Travaglini. 
The authors consider point distributions in general sub-regions 
in $\MMM$ and estimate the $L_p$-discrepancies in the terms 
of regularity of the boundary of such sub-regions, see 
\cite[Theorems~ 8.1 and~8.5]{ref5*}. 

Let us discuss in more details the results of the paper \cite{ref5*} 
related to the $L_p$-discrepancies for the metric balls $B(y,r)\subset\MMM$.
%To simplify the discussion, we
Let $\MMM$ be a compact $d$-dimensional Riemannian manifold $\MMM$, 
and  $r_{\MMM}$
the injectivity radius of $\MMM$ (for the definition, see,
for example, \cite[p. 142]{ref10}).

It is proved in \cite[Corollary~8.4]{ref5*} that for any $N$ and 
each $0<p<\infty$ there exists
an $N$-point subset $\DDD_N \subset\MMM$ such that 
\begin{equation}
\LLL_p[\MMM,\xi,\DDD_N]\,\,\lesssim \,\,N^{\frac12 - \frac{1}{2d}},
\label{eq1.9**}
\end{equation}
provided that the measure $\xi$ is concentrated on the sub-interval 
$[0,r_{\MMM})\subseteq\III$.

Similarly, it is proved in \cite[Corollary~8.6]{ref5*} that for any $N$ there exists
an $N$-point subset $\DDD_N \subset\MMM$ such that 
\begin{equation}
\LLL^{0}_{\infty}[\MMM,\DDD_N]\,\,\lesssim \,\,N^{\frac12-\frac{1}{2d}}\,
(\log N)^{1/2},
\label{eq1.12**}
\end{equation}
where $\LLL^{0}_{\infty}[\MMM,N]$ is defined by \eqref{eq1.7} but the supremum is
taken over balls $B(y,r)\subset\MMM$ with  $0\le r<r_{\MMM}$.

Notice that for the spheres $S^d$, and for some other spaces, the injectivity radius 
is equal to the diameter, and the bound \eqref{eq1.12**} implies Beck's bound 
\eqref{eq1.12*}, while the bounds \eqref{eq1.9**} are new for $S^d$ and $2<p<\infty$.
%Such strong restrictions seems to be unnecessary. 

Naturally, the question arises as to whether the parameter $r_{\MMM}$ is really
needed in the bounds \eqref{eq1.9**} and \eqref{eq1.12**}.

The occurrence of the injectivity radius in the above results has 
a pure geometric character and dictated by a local treatment of 
the discrepancies. For small radii $0\le r<r_{\MMM}$ 
the geodesic balls $B(y,r)$ are diffeomorphic to the balls in the Euclidean 
space $\Rr^d$,
while for large $r\ge r_{\MMM}$ the structure of balls $B(y,r)$ becomes
rather complicated. 

At the same time, the volume $v(y,r)$ as a function of radii
is quite regular for all $r\in\III$. Furthermore, $v(y,r)$ can be estimated
very precisely by the volume comparison theorems, well-known in the 
Riemannian geometry, see~\cite[Chapter~8]{ref9*} and \cite[Chapter~9]{ref10}. 
In the present paper we show that the volume function $v(y,r)$ suffices
to treat the discrepancies. This allows us to eliminate $r_{\MMM}$
from the bounds \eqref{eq1.9**} and \eqref{eq1.12**} and estimate  
the $L_p$-discrepancies in a much more general situation. 

We specialize metric measure spaces by the following two conditions.

\textit{Condition A.} The volume $v(y,r)$ satisfies the bounds
\begin{equation}
c_1^{-1}r^d\,\,\le v(y,r)\,\,\le c_1 r^d, \quad y\in \MMM, \,\, r\in \III,
\label{eq1.2}
\end{equation}
with positive constants $d$ and $c_1$ independent of $y\in \MMM$ and $r\in \III$

The spaces satisfying the Condition A are known as Ahlfors regular spaces. 
%see, for example, 
%\cite{ref7}. 

In the following, we write consecutively $c_1, c_2, c_3, \dots$ for positive constants
depending only on $\MMM$. 

\textit{Condition B.} The volume $v(y,r)$ as a function of $r$ is Lipschitz continuous:
\begin{equation}
\vert v(y,r_1)-v(y,r_2) \vert\,\, \le \,\,c_2 \vert r_1-r_2 \vert,\quad y\in\MMM, 
\,\, r_1,r_2 \in \III.
\label{eq1.3}
\end{equation}

It is not difficult to give many examples of compact spaces satisfying both 
Conditions A and B. Particularly, the following is true.
\begin{proposition}\label{prop1.1} 
Any compact $d$-dimensional Riemannian manifold satisfies the Conditions A and B.
\end{proposition}
The Condition A is well-known for compact Riemannian manifolds, see, for example,
\cite[Section~9.2]{ref9*}, while the Condition B is a little more specific, it can be 
easily derived from the Bishop--Gromov volume comparison theorem. For completeness,
we shall give a short proof of Proposition~\ref{prop1.1} in Appendix 
in Section~\ref{sec4}.

Our first result is the following. 
\begin{theorem}\label{thm1.1}
Let $\MMM$ be a compact connected metric measure space satisfying 
the Conditions A and B. Then for all  $N$,
we have 
\begin{equation}
\lambda_p[\MMM,\xi,N]\,\, \le\,\, c_3 (p+1)^{1/2}N^{\frac12 - \frac{1}{2d}},
\quad 0<p<\infty,
\label{eq1.9}
\end{equation}
where $\xi$ is an arbitrary normalized measure on $\III$.

Particularly, the bound \eqref{eq1.9} holds for any compact Riemannian 
manifold of dimension $d$.
\end{theorem}
The proof of Theorem~\ref{thm1.1} is given in Section~\ref{sec3}. In its proof, 
the well-known random $N$-point distributions will be used, 
see~\cite[pp. 237--239]{ref4}. Such random distributions are constructed in terms of partitions of the space $\MMM$ into $N$ subsets of equal measure and small diameters. The local discrepancies of such distributions can be written as sums of random independent variables, and the Marcinkiewicz--Zigmund inequality can be applied to obtain the bound \eqref{eq1.9}. These arguments 
in their background are similar to those in \cite{ref5*}. A new observation is that
the Conditions A and B are sufficient to prove the bound \eqref{eq1.9} without
any additional restrictions.

In  \eqref{eq1.9}, the dependence on the exponent $p$ is described explicitly. 
This allows us to obtain upper bounds for the extremal $L_\infty$-discrepancy.
For this purpose, we use the following \textit{a priory estimate}, which is also
of interest by itself.
\begin{proposition}\label{prop1.2} 
Let the assumptions of Theorem~\ref{thm1.1} hold. Then for an arbitrary $N$-point subset
$\DDD_N\subset\MMM$, we have
\begin{equation}
\LLL_{\infty}[\MMM,\DDD_N]\,\,\le\,\,2m^{2/p}\,\,\LLL_p[\MMM,\xi_0,\DDD_N]+c_4\,Nm^{-1/d},
\label{eq1.10}
\end{equation}
where $\xi_0$ is the standard Lebesgue measure on $\III$, while $p>1$ and integer 
$m\ge c_5$ are arbitrary parameters.
Particularly, we have
\begin{equation}
\lambda_{\infty}[\MMM,N]\,\,\le\,\,2m^{2/p}\,\,\lambda_p[\MMM,\xi_0,N]+c_4\,Nm^{-1/d}.
\label{eq1.11}
\end{equation}
\end{proposition}

The proof of Proposition~\ref{prop1.2} is given in Section~\ref{sec2}.

Comparing Theorem~\ref{thm1.1} with Propositions~\ref{prop1.2} and~1.1, we arrive at 
the following.
\begin{corollary}\label{cor1.1} 
Let the assumptions of Theorem~1.1 hold. Then for all  $N$,
we have
\begin{equation}
\lambda_{\infty}[\MMM,N]\,\,\le\,\, c_6\,N^{\frac12-\frac{1}{2d}}\,(\log N)^{1/2}.
\label{eq1.12}
\end{equation}
Particularly, the bound \eqref{eq1.12} holds for any compact Riemannian 
manifold of dimension $d$.
\end{corollary}
\begin{proof}
Putting $m=N^d$ in \eqref{eq1.11} and using \eqref{eq1.9}, we obtain
\begin{equation*}
\lambda_{\infty}[N]\,\,\le \,\,2N^{\frac{2d}{p}}\,\, \lambda_p[\xi_0,N] +c_4
\,\,\le\,\, 2c_3N^{\frac{2d}{p}}\,\, (p+1)^{1/2}N^{\frac12 - \frac{1}{2d}} +c_4 .
%\label{eq1.12}
\end{equation*}
Now, we choose $p=2d\log N$ (with the log in base 2, say) to obtain 
\begin{equation*}
\lambda_{\infty}[N]\,\,\le \,\,2c_3N^{\frac12 - \frac{1}{2d}}\,\,(2d\log N +1)^{1/2} +c_4 
\,\,\le \,\,c_6 N^{\frac12-\frac{1}{2d}}\,\,(\log N)^{1/2},
%\label{eq1.12}
\end{equation*}
%
%
%and the bound \eqref{eq1.12} follows.
that completes the proof.
\end{proof}

The present paper is organized as follows. In Section~\ref{sec2}  we describe the necessary facts 
on partitions of metric measure spaces and prove Proposition~\ref{prop1.2}. 
In Section~\ref{sec3} we describe the construction of random point distributions and prove Theorem~\ref{thm1.1}. Finally, in Section~\ref{sec4} we prove 
Proposition~\ref{prop1.1}. 

%The author is grateful to William Chen and Giacomo Gigante for
%a very useful discussion of an earlier version of this paper.

%%%%%%%%%%
%
%% SECTION 2
%
%%%%%%%%%%

\section{Partitions of metric spaces. Proof of Proposition~1.2}\label{sec2}

The following general result is due to Gigante and Leopardi \cite[Theorem~2]{ref9}.
\begin{lemma}\label{lem2.1}
Let $\MMM$ be a compact connected metric measure space satisfying 
the Condition A. Then for all sufficiently large $m> c_7$
there exists a partition $P_m=\{\PPP_j\}_1^m$ of $\MMM$ into m subsets $\PPP_j$
with the following properties
\begin{equation}
\MMM=\bigcup_{1\le j\le m}\PPP_j,
\quad
\PPP_j\cap \PPP_i=0,
\quad
j\ne i,
\quad
\mu(\PPP_j)=m^{-1},
\quad
1\le j\le m
\label{eq2.1}
\end{equation}
and
\begin{equation}
c_8^{-1}\,\, m^{-1/d}\le\diam\PPP_j\le c_8\,\, m^{-1/d},
\quad
1\le j\le m
\label{eq2.2}
\end{equation}
\end{lemma}

Partitions with such properties occur in many fields of geometry and analysis.
For special spaces, such as the spheres $S^d$, they have long been in use,
see the references in \cite{ref9}.
%In the general case, the proof of Lemma~\ref{lem2.1} given in \cite{ref9} relies on 
%a nontrivial construction of the so-called 'dyadic cubes' in Alhfors 
%regular spaces \cite{ref8}.

We wish to give some simple corollaries of Lemma~\ref{lem2.1} needed for 
the proofs of Theorem~\ref{thm1.1} and Proposition~\ref{prop1.1}.
We write $\varSigma(y,r)=\{x:\theta(x,y)=r\}$ 
for the sphere in $\MMM$ of radius $r\in \III$ centered at~ $y\in\MMM$.
For a partition $P_m=\{\PPP_j\}_1^m$ of $\MMM$, we put
\begin{equation}
\left. 
\begin{aligned}
J_m=J_m(y,r)=\{j:\varSigma(y,r)\cup\PPP_j\ne\emptyset\},
\\
K_m=K_m(y,r)=\#\{J_m(y,r)\}.
\\
\end{aligned}
\quad \right\}
\label{eq2.3}
\end{equation}
Thus, $K_m$ is the number of subsets $\PPP_j\in P_m$ entirely covering 
the sphere $\varSigma(y,r)$.
\begin{lemma}\label{lem2.2}
Let $\MMM$ be a compact connected metric measure space satisfying 
the Conditions A and B and let $P_m=\{\PPP_j\}_1^m$ be the partition 
of $\MMM$ from Lemma~\ref{lem2.1}. Then, we have
\begin{equation}
K_m(y,r)\le c_9\, m^{1-\frac{1}{d}}.
\label{eq2.4}
\end{equation}
\end{lemma}
\begin{proof}
Put $\tilde\varSigma (y,r)=\bigcup_{j\in J_m}\PPP_j$. In view of \eqref{eq2.1},
$\mu(\tilde\varSigma (y,r))=m^{-1} K_m$. From the other hand, in view of \eqref{eq2.2},
the union $\tilde\varSigma (y,r)$ is a subset 
in the spherical shell $B(y,r+c_8\,m^{-1/d})\setminus B(y,r-c_8\,m^{-1/d})$.
By the Condition B, we obtain 
\begin{equation*}
K_m\le m\left(v(y,r+c_8\,m^{-1/d}) - v(y,r-c_8\,m^{-1/d})\right)
\le 4 c_2 c_8\, m^{1-\frac{1}{d}},
\end{equation*}
that completes the proof.
\end{proof}

Introduce the following kernels
\begin{equation}
\delta^{\MMM}_m(y,z)=m\sum_{1\le j\le m} \chi(\PPP_j,y)\chi(\PPP_j,z)
\quad y,z\in\MMM,
\label{eq2.5}
\end{equation}
where $P_m=\{\PPP_j\}_1^m$ is an equal measure partition of $\MMM$,
see \eqref{eq2.1},
\begin{equation}
\delta^{\III}_m(r,u)=m\sum_{1\le i\le m} \chi(\QQQ_i,y)\chi(\QQQ_i,z)
\quad r,u\in\III,
\label{eq2.6}
\end{equation}
where $Q_m=\{\QQQ_i\}_1^m$ is the partition of $\III\setminus\{0\}$ into
the segments $\QQQ=(\frac{i-1}{m},\frac{i}{m}]$ of equal length $m^{-1}$.
We put
\begin{equation}
\delta_m(y,z;r,u)=\delta^{\MMM}_m(y,z)\,\, \delta^{\III}_m(r,u).
\label{eq2.7}
\end{equation}

The kernel \eqref{eq2.7} is non-negative and one can easily check the following 
relations
\begin{equation}
\int\!\!\!\!\int_{\MMM\times\III}\delta_m(y,z;r,u)\,\dd\mu(z)\,\dd u=1,
\label{eq2.8}
\end{equation}
\begin{equation}
\left(\int\!\!\!\!\int_{\MMM\times\III}\delta_m(y,z;r,u)^q\,\dd\mu(z)\,\dd u
\right)^{1/q}=
m^{2/p}, 
\label{eq2.9}
\end{equation}
where $1<q<\infty,\,\,1<p<\infty$ and $\frac{1}{q}+\frac{1}{p}=1$.

For the characteristic function and the volume of a ball $B(y,r)$, we consider
the following approximations (piece-wise on the partition $\PPP_m\times\III_m$)
\begin{equation}
\chi_m(B(y,r),x)=\int\!\!\!\!\int_{\MMM\times\III}\delta_m(y,z;r,u)\chi(B(z,u),x)
\,\dd\mu(z)\,\dd u,
\label{eq2.10}
\end{equation}
\begin{equation}
v_m(y,r)=\int\!\!\!\!\int_{\MMM\times\III}\delta_m(y,z;r,u)v(z,u)
\,\dd\mu(z)\,\dd u.
\label{eq2.11}
\end{equation}
\begin{lemma}\label{lem2.3}
Let the assumptions of Lemma~2.2 hold. Then, we have
\begin{equation}
\chi_m(B(y,r-\varepsilon_m),x)\le\chi(B(y,r),x)\le\chi_m(B(y,r+\varepsilon_m),x),
\label{eq2.12}
\end{equation}
\begin{equation}
v_m(y,r-\varepsilon_m)\le v(y,r)\le v_m(y,r+\varepsilon_m),
\label{eq2.13}
\end{equation}
where $\varepsilon_m = 2c_8 \, m^{-1/d}$.
\end{lemma}
\begin{proof}
By the triangle inequality, the ball $B(z,u)$ contains the ball $B(y,r^-)$
with $r^-=u-\theta(y,z)$ and is contained in the ball $B(y,r^+)$ with
$r^+=u+\theta(y,z)$.

From the definitions \eqref{eq2.5} and \eqref{eq2.6}, we conclude that the 
kernel \eqref{eq2.7} does not vanish, if and only if both centers 
$y$ and $z$ belong to the same subset $\PPP_j\in P_m$ and both radii
$r$ and $u$ belong to the same subset $\QQQ_i\in Q_m$. In such a situation, 
from \eqref{eq2.2} and the definition of the partition $P_m$, we obtain
\begin{equation*}
\begin{aligned}
r^-\ge r-c_8\, m^{-1/d}-m^{-1/d}\ge r-\varepsilon_m,
\\
r^+\le r+c_8\, m^{-1/d}+m^{-1/d}\le r+\varepsilon_m.
\\
\end{aligned}
\end{equation*}
Therefore, the ball $B(z,u)$ contains the ball $B(y,r-\varepsilon_m)$ and is
contained in the ball $B(y,r+\varepsilon_m)$. For the characteristic functions,
this means
\begin{equation*}
\chi(B(y,r-\varepsilon_m),x)\le\chi(B(z,u),x)\le\chi(B(y,r+\varepsilon_m),x).
\end{equation*}
Substituting these inequalities into \eqref{eq2.10} and using \eqref{eq2.8}, we obtain
\begin{equation*}
\chi(B(y,r-\varepsilon_m),x)
\le\chi_m(B(y,r),x)\le\chi(B(y,r+\varepsilon_m),x).
\end{equation*}
Replacing in these inequalities $r$ with $r-\varepsilon_m$ and next with 
$r+\varepsilon_m$, we obtain \eqref{eq2.12}. Integrating \eqref{eq2.12}
with respect to $x\in\MMM$, we obtain \eqref{eq2.13}.
\end{proof}

\begin{proof}[Proof of Proposition~\ref{prop1.2}]
Substituting \eqref{eq2.12} and \eqref{eq2.13} into \eqref{eq1.4}, we obtain 
\begin{equation}
\begin{aligned}
L_m[B(y,r-\varepsilon_m),\DDD_N]-N\alpha^-_m(y,r)
&\le L[B(e,r),\DDD_N]
\\
&\le L_m[B(y,r+\varepsilon_m),\DDD_N]-N\alpha^+_m(y,r),
\\
\end{aligned}
\label{eq2.14}
\end{equation}
where
\begin{equation}
\left.
\begin{aligned}
\alpha^-_m(y,r)= v(y,r)-v(y,r-\varepsilon_m)\ge 0,
\\
\alpha^+_m(y,r)= v(y,r+\varepsilon_m)-v(y,r)\ge 0
\\
\end{aligned}
\right\}
\label{eq2.15}
\end{equation}
and
\begin{equation}
\begin{aligned}
&L_m[B(y,r),\DDD_N]
=\sum_{x\in\DDD_N}\chi_m(B(y,r),x)-Nv_m(y,r)
\\
&=\int\!\!\!\!\int_{\MMM\times\III}\delta_m(y,z;r,u)L[(B(z,u),\DDD_N].
\,\dd\mu(z)\,\dd u,
\end{aligned}
\label{eq2.16}
\end{equation}
From \eqref{eq2.14}, we obtain the bound
\begin{equation}
\begin{aligned}
|L[B(y,r),\DDD_N]|&\le
|L_m[B(y,r-\varepsilon_m),\DDD_N]|+|L_m[B(y,r+\varepsilon_m),\DDD_N]|
\\
&
+N\alpha^-_m(y,r)+N\alpha^+_m(y,r).
\end{aligned}
\label{eq2.17}
\end{equation}
The quantities \eqref{eq2.15} can be easily estimated by the Condition B
\begin{equation}
\alpha^-_m(y,r)\le 2c_2\,c_8\, m^{-1/d},
\quad \alpha^+_m(y,r)\le 2c_2\,c_8\, m^{-1/d}.
\label{eq2.18}
\end{equation}
Applying H\"older's inequality to the integral \eqref{eq2.16}
and using \eqref{eq2.9}, we obtain 
\begin{equation}
|L[B(y,r),\DDD_N]|\le m^{2/p}\,\,\LLL_p[\xi_0,\DDD_N],
\label{eq2.19}
\end{equation}
where $\xi_0$ is the standard Lebesgue measure on the interval $\III$.

Notice that the right hand sides in \eqref{eq2.18} and \eqref{eq2.19} are 
independent of $y$ and $r$. Substituting \eqref{eq2.18} and \eqref{eq2.19}
into \eqref{eq2.17} and using the definition of $L_{\infty}$-discrepancy
\eqref{eq1.7}, we obtain 
\begin{equation}
\LLL_{\infty}[\DDD_N]\le 2m^{2/p}\,\,\LLL_p[\xi_0,\DDD_N]+4c_2\,c_8\,N m^{-1/d}.
\label{eq2.20}
\end{equation}
This proves the bound \eqref{eq1.10} with $c_4=4c_2c_8$
and $c_5=c_7$. 
\end{proof}

%%%%%%%%%%%%%%
%
%
%%%%%%%%%%%%%%%%%%%

%%%%%%%%%%
%
% SECTION 3
%
%%%%%%%%%%

\section{Random point distributions. Proof of Theorem~1.1}\label{sec3}

Random $N$-point distributions can be constructed as follows.
Suppose that a partition $\PPP_N=\{\PPP_j\}^N_1$ of the space $\MMM$ 
into $N$ parts $\PPP_j\subset\MMM$ of equal measure $N^{-1}$ is given. 
Introduce the probability space
\begin{equation}
\Omega_N
=\prod_{1\le j\le N}\PPP_j
=\{X_N=(x_1,\ldots,x_N):x_j\in \PPP_i,1\le i\le N\},
\label{eq3.1}
\end{equation}
with a probability measure
%
%
%\begin{displaymath}
$\omega_N=\prod_{1\le j\le N} \tilde\mu_j$,
%\end{displaymath}
%
%
where $\tilde\mu_j=N\mu\vert_{\PPP_j}$,
and $\mu\vert_{\PPP_j}$ denotes the restriction of the measure $\mu$ to a subset 
$\PPP_j\subset\MMM$.
We write $\Ee F[\,\cdot\,]$ for the expectation of a random variable $F[X_N]$, 
$X_N\in\Omega_N$:
\begin{align}
\Ee F[\,\cdot\,]
&
=\int_{\Omega_N}F[X_N]\,\dd\omega_N
\nonumber
\\
&
=N^N\int\!\ldots\!\int_{\PPP_1\times\ldots\times \PPP_N}F(x_1,\ldots,x_N)\,\dd\mu(x_1)\ldots\dd\mu(x_N).
\label{eq3.2}
\end{align}
Particularly, if $F[X_N]=f(x_j)$, where $j$ is a fixed index and $f(x),\, x\in\MMM,$
is a summable function, then
\begin{equation}
\Ee F[\,.\,] = N\int_{\PPP_j} f(x)\,\dd\mu(x).
\label{eq3.3}
\end{equation}

Elements $X_N=(x_1,\ldots,x_N)\in\Omega_N$ can be thought of as random $N$-point distributions in the space~$\MMM$, and their discrepancies $\LLL_p[\xi,X_N]$
as random variables on the probability space~$\Omega_N$. We shall prove the following 
\begin{lemma}\label{lem3.1}
Let $\MMM$ be a compact connected metric measure space satisfying 
the Conditions A and B, and let the probability space~$\Omega_N$ in \eqref{eq3.1}
be constructed by the partition $P_N=\{\PPP_j\}_1^N$
of $\MMM$ from Lemma~\ref{lem2.1} with $m=N$. Then, we have
\begin{equation}
\left(\,\Ee\, |\,\LLL_p[\xi,.]\,|^p\,\right)^{1/p} \le c_{10}\, (p+1)^{1/2}\,N^{\frac12 - \frac{1}{2d}},
\quad 0<p<\infty,
\label{eq3.4}
\end{equation}
where $\xi$ is an arbitrary normalized measure on $\III$.
\end{lemma} 

Theorem~1.1 is a direct corollary of Lemma~3.1.

\begin{proof}[Proof of Theorem~\ref{thm1.1}]
It follows from \eqref{eq3.4} that for each $0<p<\infty$ there exists an $N$-point
subset $X^{(p)}_N\in\Omega_N$ such that
\begin{equation*}
\LLL_p[\xi,X^{(p)}_N] \le c_{10}\, (p+1)^{1/2}\,N^{\frac12 - \frac{1}{2d}},
\end{equation*}
and the bound \eqref{eq1.9} follows for $N>c_7$ with $c_3 = c_{10}$, while for  
$N\le c_7$, we  have $\lambda_p[\xi,N]\le 2c_7$. This proves Theorem~1.1.
\end{proof} 

For the proof of Lemma~\ref{lem3.1}, we need the Marcinkiewicz--Zigmund inequality,
which can be stated as follows.
\begin{lemma}\label{lem3.2}
Let $\zeta_j,\, j\in J, \#\{J\}<\infty,$ be a finite collection of real-valued independent random variables on a probability space $\Omega$ with expectations 
$\Ee\,\zeta_j =0, j\in J$. Then, we have 
\begin{equation}
\Ee\, |\,\sum_{j\in J}\zeta_j\,|^p\le 2^p\, (p+1)^{p/2}\,\Ee\,
(\,\sum_{j\in J}\zeta^2_j \,)^{p/2},
\quad 1\le p<\infty.
\label{eq3.5}
\end{equation}
\end{lemma} 
The proof of Lemma~\ref{lem3.2} can be found in~\cite[Section~10.3,Theorem~2]{ref7}. 
\begin{proof}[Proof of Lemma~\ref{lem3.1}]
Introduce the notation
\begin{equation}
\left. 
\begin{aligned}
J^0_m=J^0_m(y,r)=\{j:\PPP_j\subset B(y,r)\},
\\
K^0_m=K^0_m(y,r)=\#\{J^0_m(y,r)\}.
\\
\end{aligned}
\quad \right\}
\label{eq3.6}
\end{equation}
In the notation \eqref{eq2.3} and \eqref{eq3.6} the characteristic function and volume
of a ball can be written as 
\begin{equation*}
\begin{aligned}
\chi(B(y,r),x)&=\sum_{j\in J^0_N}\chi(\PPP_j,x)+\sum_{j\in J_N}\chi(B(y,r)\cap\PPP_j,x),
\\
v(y,r)&=N^{-1}K^0_N+\sum_{j\in J_N}\mu(B(y,r)\cap\PPP_j).
\\
\end{aligned}
\end{equation*}
With the help of these formulas we can calculate the local discrepancy \eqref{eq1.4}
for the random point distribution $X_N=(x_1,\ldots,x_N)\in\Omega_N$:
\begin{equation*}
\begin{aligned}
&L[B(y,r),X_N]=\#(B(y,r)\cap X_N)-N v(y,r))
\\
&=K^0_N+\sum_{j\in J_N}\chi(B(y,r)\cap\PPP_j,x_j)-K^0_N-N\sum_{j\in J_N}
\mu(B(y,r)\cap\PPP_j)
\\
&=\sum_{j\in J_N}\chi(B(y,r)\cap\PPP_j,x_j)-N\sum_{j\in J_N}\mu(B(y,r)\cap\PPP_j),
\\
\end{aligned}
\end{equation*}
and we can write
\begin{equation}
L[B(y,r),X_N]=\sum_{j\in J_N}\zeta[X_N],
\label{eq3.7}
\end{equation}
where
\begin{equation}
\zeta_j [X_N]=\zeta_j [y,r,X_N]=\chi(B(y,r)\cap\PPP_j,x_j) -N\mu(B(y,r)\cap\PPP_j)
\label{eq3.8}
\end{equation}
are random variables on the probability space $\Omega_N$.

The random variables \eqref{eq3.8} are independent, $|\zeta_j[X_N]|<1$ and, in view of
\eqref{eq3.3}, their expectations $\Ee\zeta_j[\,.\,]=0,\, j\in J_N$. Hence, 
the Marcinkiewicz--Zigmund inequality \eqref{eq3.5} can be applied to the sum
\eqref{eq3.7}, and taking the bound \eqref{eq2.4} into account, we obtain
\begin{equation}
\begin{aligned}
\Ee\, |\,\sum_{j\in J_N}\zeta_j [\,.\,]\,|^p &\le 2^p (p+1)^{1/2} K_N ^{p/2}
\\
&\le 2^p\, (p+1)^{1/2}\,\, c^{p/2}_9\,\, N^{(\frac{1}{2} - \frac{1}{2d})p}, 
\quad 1\le p<\infty.
\\
\end{aligned}
\label{eq3.9}
\end{equation}
Notice that the right hand side in \eqref{eq3.9} is independent of $y$ and $r$.
Integrating the inequality \eqref{eq3.9} with respect to the measure 
$\mu\times\xi$ on $\MMM\times\III$, we obtain
\begin{equation}
\begin{aligned}
&\int\!\!\!\!\int_{\MMM\times\III}\Ee\, |\,\sum_{j\in J_N}\zeta_j [y,r,\,.\,]\,|^p
\,\dd\mu(y)\,\dd u
\\
&=\Ee (\LLL_p[\xi,.])^p\,
\le \,2^p\, (p+1)^{1/2}\,\, c^{p/2}_9\,\, N^{(\frac{1}{2} - \frac{1}{2d})p}, 
\quad 1\le p<\infty.
\\
\end{aligned}
\label{eq3.10}
\end{equation}
This proves the bound \eqref{eq3.4} for $1\le p<\infty$. Since the left hand side
in \eqref{eq3.4} is a non-decreasing function of $p$, the bound \eqref{eq3.4}
holds for all $0\le p<\infty$.
\end{proof}

%%%%%%%%%%%%%%%%%%%%%%%%%%%%%%%%%%%%%%%%%%%%%%%%%%%%%
%
%
%%%%%%%%%%%%%%%%%%%%%%%%%%%%%%%%%%%%%%%%%%%%%%%%%%%%%

%%%%%%%%%%
%
% SECTION 4
%
%%%%%%%%%%

\section{Appendix: Proof of Proposition~1.1}\label{sec4}
In this Section we consider a compact $d$-dimensional Riemanniam manifold $\MMM$
with the standard Riemannian geodesic distance $\theta$ and measure $\mu$ defined
by the corresponding metric tensor on $\MMM$, see, \cite{ref10}. Notice that for such  $\theta$ and $\mu$ the normalization \eqref{eq1.1} fails but this is of no importance for the present discussion, because the choose of normalization has effect only on the constants in the bounds \eqref{eq1.2} and \eqref{eq1.3}.
We keep the same notation
$v(y,r)$ for the volume of a ball with respect to the metric $\theta$ and 
the measure $\mu$.

The bounds \eqref{eq1.2} for a compact Riemanniam manifold are well-known, see,
for example, \cite{ref9*}.  Recall that local consideration of any 
Riemannian manifold shows that at each point $y \in \MMM$ 
and for small $r,\,\, 0\le r <r_{\MMM}$, one has the asymptotic  
$v(y,r)=\kappa_d r^d + O(r^{d-1})$, where $\kappa_d$ is the volume of unit ball 
in $\Rr^d$, see \cite[Section~9.2]{ref9*}. 
This implies the bounds \eqref{eq1.2} for small radii $r$. 
Since $\MMM$ is compact, the bounds
\eqref{eq1.2} can be easily extended to all $0<r\le \diam \MMM$.

In order to prove the bound \eqref{eq1.3}, we compare $v(y,r)$ with the volume
$v_k(r)$ of 
a geodesic ball in the $d$-dimensional simply connected hyperbolic space of constant
negative sectional curvature $-k^2$. The volume $v_k(r)$ is independent of 
the position of its center and is given explicitly by
\begin{equation*}
v_k(r)=\sigma_d \int^r_0 \left(\frac{\sinh ku}{k}\right)^{d-1}\,\dd u,
\quad 0\le r < \infty,
%
%\label{eq4.1}
\end{equation*}
where $\sigma_d$ is the $(d-1)$-dimensional area of the unit sphere in $\Rr^d$.
\begin{lemma}\label{lem4.1}
For any compact Riemannian manifold $\MMM$, there exists a constant $k_{\MMM}\ge 0$
depending only on $\MMM$, such that for all $k> k_{\MMM}$ the ratio 
$\frac{v(y,r)}{v_k(r)}$ as a function of $r$ is non-increasing and tends to 1 
as $r\to 0$.
\end{lemma} 
Lemma~4.1 is a very special case of the Bishop--Gromov volume comparison theorem,
see \cite[Section~8.7, Theorem~8.45]{ref9*} and  \cite[Chapter~9, Lemma~36]{ref10}. The constant $k_{\MMM}$ is
the smallest $k_0\ge 0$ such that the matrix $R(y)+k^2_0(d-1)I_d$ is
not-negative defined for all $y\in \MMM$, here $R(y)$ is the Ricci tensor at
$y\in \MMM$ and $I_d$ is the identity $d\times d$ matrix.

By Lemma~4.1, for $0< r_1\le r_2\le \diam \MMM$, we have 
\begin{equation*}
\frac{v(y,r_2)}{v_k(r_2)} \le \frac{v(y,r_1)}{v_k(r_1)} \le 1.
\end{equation*}
Therefore,
\begin{equation*}
v(y,r_2)-v_k(r_1) \le \frac{v(y,r_1)}{v_k(r_1)} (v_k(r_2)-v_k(r_1) ) 
\le v_k(r_2)-v_k(r_1),
\end{equation*}
and the bound \eqref{eq1.3} follows, since $v_k(r)$ is smooth and increasing.

The proof of Proposition~1.1 is competed.

\end{document}